\documentclass[12pt,twoside]{amsart}
\usepackage{amssymb}
\usepackage{a4wide}
\usepackage[latin1]{inputenc}
\usepackage[T1]{fontenc}
\usepackage{times}

\def\hpq0{h^{p,q}_{\leq 0}}
\def\Hpq0{\H_{\leq 0}^{p,q}}

\def\dbar{\bar\partial}

\def\C{{\mathbb C}}

\def\D{\mathcal{D}}

\def\supp{{\rm supp\,}}

\def\H{{\mathcal H}}

\def\Re{{\rm Re\,  }}

\def\be{\begin{equation}}
\def\ee{\end{equation}}

\newtheorem{thm}{Theorem}[section]
\newtheorem{lma}[thm]{Lemma}

\newtheorem{prop}[thm]{Proposition}

\theoremstyle{definition}

\theoremstyle{remark}

\newtheorem{preremark}{Remark}
\newtheorem{preex}{Example}

\numberwithin{equation}{section}

\title[]
{A proof of the Ohsawa-Takegoshi theorem with sharp estimates}

\email{ bob@chalmers.se\\lempert@purdue.edu}

\author[]{ Bo Berndtsson and L\'aszl\'o Lempert}

\begin{document}

\begin{abstract}We show how  ideas from \cite{2Blocki} and \cite{3Blocki}  to prove the Suita conjecture can be adapted to give a proof of  the Ohsawa-Takegoshi extension theorem with sharp estimates.
\end{abstract}
\maketitle

\section{Introduction}

The Suita conjecture, see \cite{Suita},  states a sharp lower bound for the Bergman kernel of a plane domain in terms of the Robin constant. It was proved by B\l ocki  in  \cite{1Blocki}. The  proof in \cite{1Blocki} was based on a sharp version of the Ohsawa-Takegoshi extension theorem (\cite{1Ohsawa}), a point of view that had been introduced and advocated by Ohsawa, \cite{Ohsawa}. This was later  generalized by Guan and Zhou, \cite{Guan-Zhou}, who found a very general form of the Ohsawa-Takegoshi 
theorem with sharp constant, and also gave conditions for equality in the Suita problem. Later, in \cite{3Blocki}, B\l ocki gave a second, much simpler, proof of the Suita conjecture, based on variation of domains and the tensor power trick to get the optimal estimate. In connection with this, the second author  proposed yet another approach for the Suita conjecture which is sketched in \cite{2Blocki}, using plurisubharmonic variation of the Bergman kernel, from \cite{Yamaguchi},  \cite{Berndtsson}. This last proof is surpringly short. The aim of this note is to show how this method can be developed  to give a proof of rather general versions of the Ohsawa-Takegoshi theorem.  Apart from being quite simple, the method has the advantage of giving sharp estimates almost automatically and it might be useful in other contexts as well. It is also interesting to note that Guan and Zhou show that, conversely, their sharp version of the Ohsawa-Takegoshi theorem   gives a proof of the theorem on  variation of Bergman kernels mentioned above.

In  section 2 we  give the details of the proof from \cite{2Blocki}  of Suita's conjecture. In the third section we show how this idea can be adapted to prove the Ohsawa-Takegoshi extension theorem. The argument uses, instead of Bergman kernels, a theorem on positivity of direct image bundles from  \cite{2Berndtsson}. 
Throughout the paper we treat only domains in $\C^n$. Similar results hold also for Stein manifolds and for holomorphic sections of line bundles instead of holomorphic functions. The proofs here work in almost the same way in this more general setting but we have chosen to restrict to the case of domains in $\C^n$ in order to emphasize the basic ideas.  

\section{ Suita's conjecture}

Let $D$ be a (say smoothly bounded)  domain in the complex plane containing the origin. We denote by $K(z)$ the Bergman kernel for $A^2(D)$ restricted to the diagonal. Let $G(z)$ be the Green's function for $D$ with pole at 0. Then, 
$$
G(z)=\log |z|^2 -h(z)
$$
where $h$ is a harmonic function chosen so that $G$ vanishes on the boundary of $D$. Then $h(0):= c_D$ is the Robin function at 0. Suita's conjecture says that
$$
K(0)\geq \frac{e^{-c_D}}{\pi}.
$$
Here is the proof of this that is sketched  in  \cite{2Blocki}. Let for $t\leq 0$ 
$$
D_t=\{z\in D; G(z)<t\}.
$$
Let $K_t$ be the Bergman kernel for $D_t$ (on the diagonal). Since 
$$
\D:=\{(\tau,z); G(z)-\Re\tau<0\}
$$
is pseudoconvex in $\C^2$, it follows from \cite{Yamaguchi},  \cite{Berndtsson} that $\log K_t(0)$ is a convex function of $t$. 

When $t$ is very large negative, $D_t$ is a small neighbourhood of 0. On this small neighbourhood, $h(z)$ is almost constant so $|h(z)-c_D|<\epsilon$. Hence, if $\Delta_r$ is the disk with center 0 and radius $r$, 
$$
\Delta_{r_0}\subset D_t\subset \Delta_{r_1}
$$
if $r_0= e^{t/2+c_d-\epsilon}$ and $r_1= e^{t/2+c_d+\epsilon}$. By the monotonictity of Bergman kernels with respect to domains
$$
K_t(0)\sim \frac{e^{-t -c_D}}{\pi}
$$
as $t\to -\infty$. 
Hence
$$
k(t):=\log K_t(0) +t
$$
is in particular bounded from above as $t\to -\infty$. Since $k$ is convex, this implies that $k$ is increasing on the negative half axis. Therefore $k(0)\geq \lim_{t\to -\infty}k(t)$, so
$$
K_0(0)\geq \frac{e^{-c_D}}{\pi},
$$
which ends the proof. 

The same proof gives, as in \cite{1Blocki}, a precise estimate for the Bergman kernel for domains in higher dimensions. One then replaces the Green's function used above by any negative plurisubharmonic function $\psi$ in $D$ having a singularity like $\log|z|^2$ at the origin, and defines
$$
c_{D\psi}=\limsup_{z\to 0} \log|z|^2-\psi(z).
$$
The optimal choice of $\psi$ here is to make it as big as possible, i. e. to take $\psi=G_D$, the pluricomplex Green's function for $D$ with pole at 0. Finally, similar estimates hold for Bergman kernels with a plurisubharmonic weight.

\section{ The extension theorem}

We will now consider one setting of the Ohsawa-Takegoshi theorem. Let $D$ be a pseudoconvex domain in $\C^n$ and let $V$ be a submanifold of codimension $k$ in $D$. Let $\phi$ be plurisubharmonic in $D$. Denote 
$$
A^2(D)=\{f\in H(D); \int_D |f|^2 e^{-\phi}<\infty\}
$$
and 
$$
A^2(V)=\{f\in H(V); \int_V |f|^2 e^{-\phi}<\infty\}.
$$
Above we integrate against the volume element of the euclidean metric on $D$, and the volume element induced by the euclidean metric on $V$ respectively. 
We seek to find an extension operator  from $A^2(V)$ to $A^2(D)$ bounded by a good constant. Since the constants will be universal we can make some standard  reductions, so we assume that $D$ is smoothly bounded and strictly pseudoconvex, that $\phi$ is smooth and plurisubharmonic in a neighbourhood of the closure of $D$, and that $V$ extends as a submanifold across the boundary of $D$ as well.  All this can be achieved by replacing $D$ by a relatively compact subdomain.

\bigskip

Let $H(D)$ and $H(V)$ be the spaces of holomorphic functions in $D$ and $V$ respectively, and let  $I(V)$ be the subspace of $H(D)$ of functions vanishing on $V$. Since any function in $H(V)$ has an extension to a function in $H(D)$, we see that the restriction operator induces an isomorphism
$$
r:H(D)/I(V)\to H(V).
$$
Via this isomorphism we get two norms on (subspaces of) $H(V)$. The first norm is the $L^2$-norm 
$$
\|f\|^2_{A^2(V)}= \int_V |f|^2 e^{-\phi},
$$
(we will modify this norm somewhat later). The second norm is the norm induced by the isomorphism $r$ and the $L^2$-norm on $A^2(D)$
$$
\|f\|^2_0=\min_{F\in A^2(D); F=f \,\text{on}\, V\}} \int_D |F|^2 e^{-\phi}.
$$
In case the latter norm is finite, it is exactly the minimal norm of an extension of $f$ to $D$, and it is attained by the unique function that equals $f$ on $V$ and is orthogonal to $I^2(V):=I(V)\cap A^2(D)$. 

We want to estimate the second norm by the first one and the idea is to follow the same strategy as in section 2. Instead of estimating the norms directly, we  estimate the dual norms, and we will do this by considering a family of intermediate domains $D_t$ for $t\in (-\infty, 0)$. The estimate we are looking for is  easy for $t<<0$, since we can take an almost arbitrary extension if $t$ is very small, and we  then use a monotonicity property of the intermediate norms to conclude that the estimate also holds for $t=0$. There are however two technical problems in implementing this scheme: First, the two norms we are dealing with are not in general defined ( i. e. finite) on the same subspaces of $H(V)$.  The first norm is in general not dominated by  the second norm as there are functions in $A^2(D)$ whose restriction to $V$ are not square integrable on $V$. On the other hand, the second norm {\it is} dominated by a constant times the first norm, but that is a special case of what we want to prove, so it is preferable not to use this fact. 
The second problem comes from  the  substitute for the plurisubharmonic variation of Bergman kernels that we are going to use:  The positivity of direct image bundles from \cite{2Berndtsson}. The way this is stated in \cite{2Berndtsson} we cannot apply it directly to a varying family of domains, so instead we will use a varying family of weight functions.
We will therefore formulate the argument slightly differently, but it is probably good to keep the general idea in mind.

\bigskip

Let $d_V(z)$ be the distance from $z$ to $V$ and let  $G$ be a negative plurisubharmonic function in $D$ which satisfies
\be
G(z)\leq \log d_V^2(z)+ A
\ee
 and
\be
G(z)\geq\log d_V^2(z)- B(z)
\ee
as $z$ goes to $V$. Here $A$ is some constant, the value of  which will not be very important, and $B$ is a continous function in $D$. Again, by restricting to a relatively compact subdomain we may assume that $B$ is bounded. Since $A$ will not appear in the final estimates, it is actually enough to assume that (3.1) holds for some $A$ in each relatively compact subdomain of $D$.  
 $D$ is a pseudoconvex domain in $\C^n$ and we let $V$ be a submanifold of codimension $k$ in $D$.
The precise estimate we will prove is the following, cf \cite{Guan-Zhou}. 
\begin{thm} Let $f$ be a function in $A^2(V)$. Then there is a function $F$ in $A^2(D)$ whose restriction to $V$ equals $f$, which satisfies
$$
\int_D |F|^2 e^{-\phi}\leq \sigma_k\int_V |f|^2 e^{-\phi +kB},
$$
where $\sigma_k$ is the volume of the unit ball in $\C^k$. 
\end{thm}

\bigskip

To prove the theorem, recall that we may assume that $V$ extends to a neighbourhood of $\bar D$, that $D$ is smoothly bounded and strictly pseudoconvex and that $B$ also extends to a continuous function in a neighbourhood of $\bar D$. We may then also assume   that $f$ extends to be holomorphic over the boundary, so that we know a priori that $f$ has {\it some } extension $F$ in $A^2(D)$. Then the optimal extension, that we call $F_0$, is the projection of $F$ to the orthogonal complement of $I^2(V)$. The norm of $F_0$ can be computed as 
\be
\|F_0\|_{A^2(D)}=\sup |\langle \xi,f\rangle|/\|\xi\|_{A^2(D)^*},
\ee
with the sup taken over all $\xi\neq 0$ in the dual of $A^2(D)$ that vanish on $I^2(V)$. Indeed, it is enough to take $\xi$ in some dense subspace of this space. We will take $\xi=\xi_g$ of the form
$$
\langle \xi,f\rangle=\sigma_k\int_V f\bar g e^{-\phi +kB},
$$
where $g$ is in $C^{\infty}_c(V)$. It is clear that such functionals are dense in the dual of $A^2(D)\ominus I^2(V)$, since if $\langle \xi_g,f\rangle=0$ for all such $g$, then $f$ must vanish on $V$. Put
$$
\|f\|_V^2:=\sigma_k \int_V |f|^2 e^{-\phi +kB}.
$$
By (3.3) 
$$
\|F_0\|_{A^2(D)}\leq \sup_g\|f\|_V\|\tilde g\|_V/\|\xi_g\|_{A^2(D)^*},
$$
if we denote by $\tilde g$ the ortogonal projection of $g$ to $A^2(V)$ with respect to the norm $\|\cdot\|_V$. It is  therefore enough to prove that
\be
\|\tilde g\|_V\leq \|\xi_g\|_{A^2(D)^*}.
\ee

We now let $D_t=\{z\in D; G(z)<t\}$ for $t<0$ and put $$
\psi(\tau,z):=\max(G(z)-\Re\tau, 0).
$$
Thus $\psi$  is a plurisubharmonic function in the left half plane times $D$, only depending on $(t,z)$ where $t=\Re\tau$. Let for $p>0$
$$
A^2_{t,p}:=\{h\in H(D); \|h\|^2_{t,p}:=\int_D |h|^2 e^{-(\phi +p\psi(t,\cdot))}<\infty\}.
$$
As vector spaces $A^2_{t,p}$  are all equal to $A^2(D)=A^2_{0,p}$ for any $p$.  As $p\to \infty$,  $ \|h\|^2_{t,p}$ tends to 
$$
\int_{D_t} |h|^2 e^{-\phi},
$$
so we can think of $A^2_{t,p}$ as a substitute for $A^2(D_t)$.
To simplify notation we consider $p$ fixed for the moment and write 
$$
\|\xi_g\|_t= \|\xi_g\|_{t, p}^*.
$$
We claim that $\log \|\xi_g\|_t$ is a convex function of $t$. This fact replaces the use of plurisubharmonic variation of Bergman kernels in section 2 and is the most important ingredient in our proof. Note that, if we take $g$ to be a point mass at a point $a$ instead of a function in $C^\infty_c$, then $\|\xi_g\|_t^2$ equals the Bergman kernel $K_t(a)$, so the convexity of  $\log \|\xi_g\|_t$ is indeed a generalization of the convexity of $\log K_t$ that we used earlier.  To see why it holds, note that it follows from the result on positivity of direct images in \cite{2Berndtsson}, Theorem 1.1, that the trivial vector bundle over the left half plane with fiber $A^2(D)$ equipped with the norms depending on $t$ (or $\tau$) defined above has positive curvature.  Hence the dual bundle has negative curvature, which means that $\log\|\xi\|_t$ is subharmonic, hence convex, if $\xi$ is any element of the dual space $A^2(D)^*$. (Strictly speaking this is only proved in \cite{2Berndtsson} when $\phi$ and $\psi$ are smooth, but of course the general case follows by approximation.) 
\begin{lma}
$$
\|\xi_g\|^2_t e^{kt}= O(1)
$$
as $t\to -\infty$. As a consequence 
$$
k_\xi(t):=\log \|\xi_g\|^2_t +kt
$$
is an increasing function of $t$. Hence 
$$
\|\xi_g\|_0^2\geq \lim_{t\to -\infty}
\|\xi_g\|^2_t e^{kt}.
$$
\end{lma}
\begin{proof}
If $h$ is a function in $A^2_t$ for $t$ large negative the submeanvalue inequality gives that
$$
\int_{V\cap \supp (g)}|h|^2 e^{-\phi+kB}\leq C' e^{-kt}\int_{D_t} |h|^2e^{-\phi}\leq C'e^{-kt}\|h\|^2_{t,p}.
$$
Hence
$$
\|\xi_g\|^2_t\leq C e^{-kt}
$$
as we wanted. This implies that
$$
k_\xi(t)=\log \|\xi_g|^2_t +kt
$$
is bounded from above as $t$ goes to minus infinity, which together with the convexity implies that $k_\xi$  is increasing.
\end{proof}
For the converse direction we use the following lemma which implies that for a fixed function $H$, say holomorphic in a neighbourhood of $V\cap\bar D$, the norms 
$\|H\|_{A^2(D_t)}$ are asymptotically majorized by $e^{kt/2}\|h\|_{A^2(V)}$, if $h$ is the restriction of $H$ to $V$. 
\begin{lma} Let $\chi$ be a continuous function on $\bar D$. Then 
$$
\limsup_{t\to-\infty}e^{-kt}\int_{D_t} \chi\leq\sigma_k \int_V \chi e^{kB}.
$$
\end{lma}
We omit the easy proof.  As a preparation for the final lemma we need a technical estimate.
\begin{lma} Let $\nu(t)$ be an increasing function for $t<0$ and assume that $\nu(t)\leq e^{kt}$. Then for $p>k$
$$
\liminf_{t\to -\infty} e^{-kt}\int_t^0 e^{-p(s-t)}d\nu(s)\leq \frac{k+1}{p-k}.
$$
\end{lma}
\begin{proof} Let
$$
f(t):= e^{-kt}\int_t^0 e^{-p(s-t)}d\nu(s).
$$
It suffices to prove that
\be
\int_{-T}^0 f(t)dt\leq T \frac{k+1}{p-k}
\ee
for $T$ large. 
But
$$
\int_{-T}^0 f(t)dt=\int\int_{-T<t<s<0}e^{-ps}e^{(p-k)t}dtd\nu(s)\leq
\frac{1}{p-k}\int_{-T}^0 e^{-ks}d\nu(s).
$$
Integrating by parts we get
$$
\int_{-T}^0 e^{-ks}d\nu(s)\leq \nu(0)+k\int_{-T}^0 ds\leq T(1+k),
$$
so we are done. 
\end{proof}

\begin{lma} For any $\delta>0$
$$
\lim_{t\to -\infty} \|\xi_g\|^2_t e^{kt}\geq \|\tilde g\|^2_V-\delta,
$$
if $p$ is large enough. (Note that the limit exists by Lemma 3.2.)
\end{lma}
\begin{proof}
The holomorphic function $\tilde g$ on $V$ can be approximated  in $L^2$ norm on $V$ by a function $ g'$ that extends holomorphically to a neighbourhood of $ \bar D$. For this we recall that $V$ extends to a submanifold $V'$ in a neighbourhood of $\bar D$ and approximate $\tilde g$ by a holomorphic function on $V'$ that extends across the boundary of $V'\cap D$. (See the appendix for a discussion why this is possible.) This function can then be extended to a neighbourhood of $\bar D$ by general Stein theory. Then
\be
\|\xi_g\|_t\geq |\sigma_k\int_V  g'\overline{\tilde g}e^{-\phi+kB}|/\| g'\|_{A^2_{t,p}}\geq (1-\epsilon)\|\tilde g\|_V^2/\| g'\|_{A^2_{t,p}}
\ee
if the approximation is good  enough. The proof will be concluded if we can prove for arbitrary $\epsilon'>0$ and with sufficiently large $p$  that 
\be
 \liminf_{t\to -\infty}e^{-kt/2}\| g'\|_{A^2_{t,p}}\leq \|\tilde g\|_V+\epsilon'
\ee
since we know that the limit in the lemma exists. But
$$
\| g'\|^2_{A^2_{t,p}}=\int_{D_t} |g'|^2e^{-\phi} +\int_{t<G}
|g'|^2e^{-\phi-p\psi}=: I+II.
$$
By Lemma 3.3 
$$
I=\int_{D_t} |g'|^2e^{-\phi}\leq (1+\epsilon) e^{kt}\| g'\|^2_V\leq (1+\epsilon)^2 e^{kt}\|\tilde g\|^2_V,
$$
if $t$ is sufficiently large and 
 $$
\| g'\|_V\leq(1+\epsilon)\|\tilde g\|_V.
$$ 
We will now prove that $II$ is small compared to $I$. For this we estimate $|g'|^2 e^{-\phi}$ by its maximum $M$ on $D$. 
Write $\nu(t)$ for the volume of $D_t$. By Lemma 3.3,  $\nu(t)\leq Ce^{kt} $ for a certain constant $C$. Then
$$
II\leq M \int_t^0 e^{-p(s-t)}d\nu(s).
$$
By Lemma 3.4
$$
II\leq MC\frac{1+k}{p-k} e^{kt}
$$
for a sequence of $t$ tending to $-\infty$.

All in all
$$
e^{-kt}\|g'\|^2_{A^2_{t,p}}\leq I + II\leq (1+\epsilon)^2\|\tilde g\|^2_V +\epsilon,
$$
if $p$ is large enough, for a sequence of  $t$ tending to $-\infty$. Hence we have proved (3.7) which together with (3.6) gives the lemma.

\end{proof}

By Lemmas 3.2 and 3.5 we now get
$$
\|\xi_g\|_{A^2(D)^*}\geq\lim_{t\to -\infty}\|\xi_g\|_t e^{kt/2}\geq \|\tilde g\|_V-\delta,
$$
for any $\delta>0$ . This proves (3.4) and therefore Theorem 3.1.

\subsection{ The adjoint formulation}
In this subsection we give  a variant of Theorem 3.1, corresponding to the so called 'adjoint formulation' of the Ohsawa-Takegoshi theorem.
This means that we think of $f$ and $F_0$ as holomorphic forms of maximal degree instead of functions. We are then given a holomorphic $(n-k,0)$-form $f$  on $V$, and we want to find a holomorphic $(n,0)$-form on $D$, $F_0$, such that 
$$
F_0=f\wedge dg_1\wedge...dg_k
$$
on $V$, if $g_1, ...g_k$ are holomorphic functions on $D$ whose common zero locus is $V$. We also assume $dg:= dg_1\wedge...dg_k$ does not vanish on $V$. 
\begin{thm} Let $D$ be a pseudoconvex domain in $\C^n$ and let $V$ be a holomorphic submanifold of $D$ defined by the equation $g=(g_1, ...g_k)=0$, where $g$ is holomorphic and $dg=dg_1\wedge...dg_k\neq 0$ on $V$. Let $\phi$ be plurisubharmonic in $D$. Assume $|g|\leq 1$ in $D$. Then there is a holomorphic $(n,0)$-form $F_0$ in $D$, such that
$$
F_0=f\wedge dg
$$
on $V$, and
$$
\int_D c_n F_0\wedge\bar F_0  e^{-\phi}\leq \sigma_k \int_V c_{n-k}f\wedge \bar f e^{-\phi},
$$
where $\sigma_k$ is the volume of the unit ball in $\C^k$.
\end{thm}
Here $c_n=(i)^{n^2}$ is a unimodular constant chosen so that $c_n F_0\wedge\bar F_0$ is positive. 
This follows in the same way as Theorem 3.1 if we let $G=\log |g|^2$. The only difference in the proof is that we replace Lemma 3.2 but the following statement. 
\begin{lma} Let $\chi'$ be a continuous $2(n-k)$-form on $\bar D$ and let  
$$
\chi=c_k \chi'\wedge dg\wedge d\bar g.  
$$ 
 Then
$$
\lim_{t\to-\infty}e^{kt}\int_{\log|g|^2<t}\chi=\sigma_k \int_V\chi'.
$$
\end{lma}

This is easily proved using a partition of unity and choosing local coordinates $(z_1, ...z_n)$ such that $(z_1, ...z_k)=(g_1, ...g_k)$.
\subsection{ A more general version}

In the version of the extension theorem we have considered so far, the submanifold $V$ was defined by a negative plurisubharmonic function $G$. In this section we will discuss a more general case when $G$ satisfies a bound $G<\psi$, where $\psi$ is another plurisubharmonic function. 
This more liberal growth condition means intuitively that the variety $V$ is allowed to be bigger. One is then still able to extend, but the extended function  needs to be larger as well.

Such situations also occur naturally  when $V$ is defined as the zero locus of a holomorphic section, $s$,  of a vector bundle $E$ over $V$, cf. \cite{Manivel}. Indeed, if we assume  the norm of $s$ with respect to a hermitian metric $h$ on $E$ is bounded by 1, and that the curvature of $h$, satisfies a bound
$$
\Theta^h\leq dd^c\psi\otimes I,
$$
then $G:=\log |s|^2_h +\psi$ will be a plurisubharmonic function satisfying $G<\psi$. This follows since
$$
dd^c \log |s|^2_h\geq -\frac{\langle \Theta^h s, s\rangle_h}{|s|_h^2}.
$$

\begin{thm} Let $G$ be a plurisubharmonic function in $D$ satisfying conditions (3.1) and (3.2), and assume $G<\psi$, where $\psi$ is plurisubharmonic in $D$.  Let $f$ be a function in $A^2(V)$. Assume that $\phi$ is plurisubharmonic in $D$ and satisfies $dd^c\phi\geq \delta dd^c\psi$, where $\delta>0$. Then there is a function $F$ in $A^2(D)$ whose restriction to $V$ equals $f$ and satisfies the estimate
$$
\int_D |F|^2 e^{-\phi -k\psi}\leq (k/\delta +1)\sigma_k\int_V |f|^2 e^{-\phi +kB}.
$$
\end{thm}
We will prove Theorem 3.7 by reducing it to Theorem 3.1. 
Let
$$
\tilde D:=\{ (z_0,z)\in \C\times D; |z_0|^2<e^{-\psi}\}.
$$
Then $\tilde D$ is pseudoconvex in $\C^{n+1}$ and there is a natural projection map $p$ from $\tilde D$ to $D$. Let $\tilde V= p^{-1}(V)$. Assume for the moment that $\psi$ is bounded by a constant $a$. Take $\epsilon>0$ and  let
$C=C(\epsilon,\psi)=\log(1+\epsilon e^a)$. Note that $C(\epsilon, \psi)\to 0$ as $\epsilon\to 0$. Let
$$
\tilde G(z_0,z):= G(z)+\log(|z_0|^2 +\epsilon)- C.
$$
Then $\tilde G$ is negative and plurisubharmonic in $\tilde D$.
Let $\tilde f(z_0,z)= z_0^k f(z)$. Take $0<\delta\leq 1$. By Theorem 3.1, with $B$ replaced by $B-\log(|z_0|^2 +\epsilon)+C$ and $\phi$ replaced by $\phi +(1-\delta)\log |z_0|^2$, there is a function $\tilde F$ in $A^2(\tilde D)$ 
that extends $\tilde f$ and satisfies
$$
\int_{\tilde D} \frac{|\tilde F|^2}{|z_0|^{2-2\delta}} e^{-\phi}\leq\sigma_k\int_{\tilde V} 
\frac{|\tilde f|^2}{|z_0|^{2-2\delta}(|z_0|^2+\epsilon)^{k}} e^{-\phi+kB}(1+\epsilon e^a)^k.
$$
A standard limiting argument shows that an extension satsifying this estimate can be found also for $\epsilon =0$ and without the assumption that $\psi$ be bounded. 
Then, replacing $\tilde F$ by 
$$
\int_0^{2\pi} \tilde F(e^{i\theta}\cdot)e^{-ki\theta} d\theta/2\pi,
$$
we see that we may assume that $\tilde F= z_0^k F(z)$. Theorem 3.7 then follows (for $\delta\leq 1$) if we carry out the integration with respect to $z_0$ and replace the arbitrary plurisubharmonic function $\phi$ by $\phi+\delta\psi$. For $\delta>1$ we write $\delta= m+\delta'$, where $m$ is an integer and $\delta'<1$, and run the same argument with $\tilde f:= z_0^{k+m} f$.

\section{Appendix}
In this appendix we state and prove the approximation result that was used in the proof of Lemma 3.4.  It could be derived from a theorem of Kerzman  (\cite{Kerzman}, Theorem 1.4.1), which deals with $L^p$-approximation,  but we include a quick proof along different lines.
\begin{prop} Let $X$ be a Stein manifold and let 
$$
\Omega=\{z\in X; \rho(z)<0\}
$$
be a relatively compact subdomain of $X$, defined by a smooth strictly plurisubharmonic function in $X$ with $d\rho\neq 0$ on $\partial\Omega$. Let $dV$ be a smooth volume form on $X$, and let $f$ be a holomorphic function on $\Omega$ such that
$$
\int_\Omega |f|^2 dV<\infty.
$$
Then there is a sequence of  functions $f_j$, holomorphic on all of $X$, such 
that
$$
\lim_{t\to \infty}\int_\Omega |f-f_j|^2 dV=0
$$
\end{prop}
\begin{proof} We first approximate $f$ by  holomorphic functions that are smooth up to the boundary of $\Omega$. For this we take a sequence of cut-off function $\chi_j$, compactly supported in $\Omega$ that increase to 1 in $\Omega$. Then $\chi_j f$ tend to $f$ in $L^2(\Omega)$, so if we denote by $P$ the Bergman projection operator for $L^2(\Omega, dV)$, $ f_j:=P(\chi_j f)$ also tend to $f$ in $L^2$. These functions are of course holomorphic, and it follows from the regularity of the $\dbar$-Neumann problem that they are smooth up to the boundary, since
$$
f_j=\chi_j f-v_j,
$$
where $v_j$ is the $L^2$-minimal solution of $\dbar v= f\dbar\chi_j$.

We may thus assume from the start that $f$ is smooth up to the boundary. We then only need to extend $f$ smoothly to a function $F$ with compact support in $X$. Solve $\dbar u_j=\dbar F$ with $L^2$-estimates for the weights
$e^{-j \max(\rho, 0)}$. Then $F- u_j$ are holomorphic on $X$ and tend to $f$ in $L^2(\Omega, dV)$, since
$$
\int_X |\dbar F|^2 e^{-j \max(\rho, 0)}dV=\int_{\rho>0} |\dbar F|^2 e^{-j\rho}dV
$$
tends to zero.
\end{proof}

\def\listing#1#2#3{{\sc #1}:\ {\it #2}, \ #3.}

\end{document}